%% file: mumford.tex
\documentclass[11pt]{amsart}

\usepackage{amsmath,amssymb,latexsym,amsthm,newlfont,enumerate}
\usepackage{hyperref}

\usepackage[all]{xy}

\input{macros}

\begin{document}
\title[Rationally connected varieties -- on a conjecture of Mumford]{Rationally connected varieties --\\ on a conjecture of Mumford}

\author{Vladimir Lazi\'c}
\address{Mathematisches Institut, Universit\"at Bonn, Endenicher Allee 60, 53115 Bonn, Germany}
\email{lazic@math.uni-bonn.de}

\author{Thomas Peternell}
\address{Mathematisches Institut, Universit\"at Bayreuth, 95440 Bayreuth, Germany}
\email{thomas.peternell@uni-bayreuth.de}

\dedicatory{Dedicated to the memory of Qikeng Lu}

\thanks{Lazi\'c was supported by the DFG-Emmy-Noether-Nachwuchsgruppe ``Gute Strukturen in der h\"oherdimensionalen birationalen Geometrie". Peternell was supported by the DFG grant ``Zur Positivit\"at in der komplexen Geometrie".
}

\begin{abstract}
We establish a conjecture of Mumford characterizing rationally connected complex projective manifolds in several cases.
\end{abstract}

\maketitle

\setcounter{tocdepth}{1}
\tableofcontents

\section{Introduction}

A complex projective manifold $X$ is \emph{rationally connected} if any two general points can be joined by a chain of rational curves. On a rationally connected manifold one can find (many) rational curves $C \subseteq X$ such that $T_X \vert_C$ is ample and deduce that 
$$ H^0\big(X,(\Omega^1_X)^{\otimes m}\big) = 0\quad\text{for all } m \geq 1.$$ 
We refer to \cite{Kol96} for a detailed discussion of rationally connected varieties. A well-known conjecture of Mumford says that the converse is also true  \cite[Conjecture IV.3.8.1]{Kol96}.

\begin{con} \label{mum}
Let $X$ be a projective manifold such that 
$$ H^0\big(X,(\Omega^1_X)^{\otimes m}\big) = 0\quad\text{for all } m \geq 1.$$
Then $X$ is rationally connected.
\end{con} 

This conjecture holds when the dimension of $X$ is at most $3$ \cite{KMM92}, and not much has been known in higher dimensions. It is, however, well known that this conjecture is equivalent to a weaker statement which says that in the context of Conjecture \ref{mum}, the variety $X$ is \emph{uniruled}, i.e.\ covered by rational curves:

\begin{con} \label{mum2}
Let $X$ be a projective manifold such that $K_X$ is pseudoeffective. Then there exists a positive integer $m$ such that 
$$ H^0\big(X,(\Omega^1_X)^{\otimes m}\big) \neq 0. $$
\end{con} 

The connection to uniruledness comes from the main result of \cite{BDPP}, stating that a projective manifold $X$ is uniruled if and only if $K_X$ is not pseudoeffective. A short proof of the equivalence of these two conjectures is given in Proposition \ref{pro:mumford}. A similar proof yields the following weaker characterization of rational connectedness obtained in \cite{Pet06,CDP15}: 

\begin{thm} 
Let $X$ be a projective manifold. Then $X$ is rationally connected if and only if for some ample line bundle $A$ on $X$ there is a constant $C$ depending on $A$ such that 
$$ H^0\big(X,(\Omega^1_X)^{\otimes m} \otimes A^{\otimes k}\big) = 0 $$
for all positive integers $k$ and $m$ with $m \geq Ck.$
\end{thm} 

We also mention a stronger conjecture from \cite{BC15}, stating that $X$ is rationally connected if and only if 
\begin{equation}\label{eq:11}
H^0\big(X, S^k\Omega^p_X \big) = 0 
\end{equation}
for all positive integers $k$ and $p$. The main result of \cite{BC15} is that condition \eqref{eq:11} implies that $X$ is simply connected. 

\medskip

In this paper, we prove several results towards Conjecture \ref{mum2}.
Notice that if $\kappa (X,K_X) \geq 0$ in Conjecture \ref{mum2}, then in particular 
$$ H^0\big(X,(\Omega^1_X)^{\otimes m}\big) \neq 0 $$
for some $m$, see Remark \ref{rem:tensor}, but Conjecture \ref{mum2} is much weaker than the nonvanishing $\kappa (X,K_X) \geq 0$. 

An important invariant of a projective manifold $X$ with pseudoeffective canonical bundle $K_X$  is its numerical dimension $ \nu(X,K_X)$. If $Y$ is a minimal model of $X$, so that $K_Y$ is nef, then $\nu(X,K_X)=\nu(Y,K_Y)$ is the largest non-negative integer $d$ such that $K_Y^d \not \equiv 0$; the general definition is in Section \ref{prelim}. It is well known that $\nu(X,K_X) \geq \kappa (X,K_X) $, and one of the equivalent formulations of the abundance conjecture is that
$$ \nu(X,K_X) = \kappa (X,K_X).$$ 
The abundance conjecture is known to hold when $\nu(X,K_X) = 0$ by \cite{Nak04} and when $\nu(X,K_X) = \dim X$ by \cite{Sho85,Kaw85b}, which in particular proves Conjecture \ref{mum2} in these cases. Thus it remains to prove Conjecture \ref{mum2} when
$$ 1 \leq  \nu(X,K_X) \leq  \dim X -1.$$ 
In this paper we deal with the extremal cases $\nu(X,K_X) = 1$ and $\nu(X,K_X) = \dim X - 1$.

Our main result is the following.

\begin{thm}\label{main}
Conjecture \ref{mum2} holds when $\dim X=4$ and $\nu(X,K_X)\neq2$.
\end{thm} 

The theorem is a consequence of much more general results which work in every dimension. We first prove Conjecture \ref{mum2} when $\nu(X,K_X)=1$ and $X$ has a minimal model, see Theorem \ref{thm:nu11}:

\begin{thm}\label{thm:1}
Let $X$ be a projective manifold such that $K_X$ is pseudo-effective, and assume that $X$ has a minimal model. If $\nu(X,K_X) =1$, then there exists a positive integer $m$ such that 
$$ H^0\big(X,(\Omega^1_X)^{\otimes m}\big) \neq 0. $$
\end{thm} 

The main input are our techniques from \cite{LP16}, where -- among other results -- we proved the abundance conjecture for varieties with $\nu(X,K_X)=1$ and $\chi(X,\OO_X)\neq0$. We discuss these ideas in Section \ref{sec:nd1}.

When $\nu(X,K_X) = \dim X -1$, we show in Theorem \ref{thm:nu-n}:

\begin{thm} 
Let $X$ be a minimal terminal $n$-fold with $\nu(X,K_X) = n-1$ and $n \geq 4$, and let $\pi\colon Y \to X$ be a resolution which is an isomorphism over the smooth locus of $X$. Assume one of the following:
\begin{enumerate}
\item[(a)] $(\pi^*K_X)^{n-2}  \cdot c_2(Y) \neq 0$;
\item[(b)] $(\pi^*K_X)^{n-2}  \cdot c_2(Y) = 0$ and $(\pi^*K_X)^{n-3} \cdot K_Y \cdot c_2(Y) \ne 0$. 
\end{enumerate} 
Then $K_X$ is semiample. 
\end{thm} 

The result is more precise when $n=4$, see Theorem \ref{thm:nu3}, which then implies Theorem \ref{main} when $\nu(X,K_X)=3$. The proof is by a careful analysis of the Hirzebruch-Riemann-Roch for two different sets of line bundles, together with a well-known slight refinement of the Kawamata-Viehweg vanishing.

Finally, we note that results from \cite{LP16} immediately give the following.

\begin{thm} 
Let $X$ be a projective manifold of dimension $n$ with $K_X$ pseudoeffective. Assume that $K_X$ has a metric with algebraic singularities and semipositive curvature current.
\begin{enumerate}
\item[(i)] If good minimal models for klt pairs exist in dimensions at most $n-1$, then there is a positive integer $m$ such that
$$ H^0\big(X,(\Omega^1_X)^{\otimes m}\big) \neq 0. $$
\item[(ii)] If $n=4$, then there is a positive integer $m$ such that
$$ H^0\big(X,(\Omega^1_X)^{\otimes m}\big) \neq 0. $$
\end{enumerate} 
\end{thm} 

This is Theorem \ref{thm4} below. It is expected that the assumptions in the theorem always hold, see Remark \ref{rem:metric}.

All the results of this paper apply also in the context of a stronger conjecture from \cite{BC15} mentioned above.
 
\section{Prelimimaries} \label{prelim}

We work over the complex numbers, and all varieties are normal and projective. For the basic notions of the Minimal Model Program we refer to \cite{KM98}. In particular, a normal projective variety is {\it terminal} if it has terminal singularities. We shortly review the definition of the numerical dimension of a pseudoeffective divisor \cite{Nak04,Kaw85}; we are mostly interested in the case when the divisor is $K_X$. 

\begin{dfn}\label{dfn:kappa}
Let $X$ be a smooth projective variety and let $D$ be a pseudoeffective $\Q$-divisor on $X$. If we denote
$$\sigma(D,A)=\sup\big\{k\in\N\mid \liminf_{m\rightarrow\infty}h^0(X, \mathcal O_X(\lfloor ( mD\rfloor+A))/m^k >0\big\}$$
for a Cartier divisor $A$ on $X$, then the {\em numerical dimension\/} of $D$ is
$$\nu(X,D)=\sup\{\sigma(D,A)\mid A\textrm{ is ample}\}.$$
Note that this coincides with various other definitions of the numerical dimension by \cite{Leh13,Eck16}. 
If $X$ is a projective variety and if $D$ is a pseudoeffective $\Q$-Cartier $\Q$-divisor on $X$, then we set $\nu(X,D)=\nu(Y,f^*D)$ 
for any birational morphism $f\colon Y\to X$ from a smooth projective variety $Y$. 
\end{dfn}

If $D$ is nef, then $\nu(X,D)$ is the largest positive integer $e$ such that $ D^e  \not \equiv 0$. Using a refined intersection theory, this can be generalized to pseudoeffective divisors \cite{BDPP}. One of the most important properties we use is that the numerical dimension is preserved by the operations of a Minimal Model Program.

The following well-known result \cite[Corollary 10.38(2)]{Kol13} is a consequence of the usual Kawamata-Viehweg vanishing theorem. The proof is analogous to \cite[Corollary]{Kaw82}; we include a short argument for the convenience of the reader. 

\begin{lem}\label{lem:KVvanishing}
Let $(X,\Delta)$ be a $\Q$-factorial projective klt pair of dimension $n$ and let $D$ be a Cartier divisor on $X$ such that $D\sim_\Q K_X+\Delta+L$, where $L$ is a nef $\Q$-divisor with $\nu(X,L)=k$. Then
$$H^i\big(X,\OO_X(D)\big) =0\quad\text{for all }i>n-k.$$
\end{lem}

\begin{proof}
The proof is by induction on $n$. If $k=n$, then this is the usual Kawamata-Viehweg vanishing \cite[Theorem 1-2-5 and Remark 1-2-6]{KMM87}. Now, assume that $k<n$ and let $H$ be an irreducible very ample divisor on $X$ which is general in the linear system $|H|$. Consider the exact sequence
\begin{equation}\label{eq:seq}
0\to\OO_X(D)\to\OO_X(D+H)\to\OO_H(D+H)\to 0.
\end{equation}
For $i>n-k$ we have $H^i\big(X,\OO_X(D+H)\big)=0$ by Kawamata-Viehweg vanishing. Since 
$$(D+H)|_H\sim_\Q K_H+\Delta|_H+L|_H$$
by the adjunction formula, see e.g.\ \cite[Proposition 4.5]{Kol13}, since the pair $(H,\Delta|_H)$ is klt by \cite[Lemma 5.17]{KM98} and since $\nu(H,L|_H)=k$, we have 
$$H^{i-1}\big(H,\OO_H(D+H)\big)=0$$ 
by induction. Then the result follows from the long exact sequence in cohomology associated to \eqref{eq:seq}.
\end{proof}

We frequently use the following theorem \cite[Theorem 4.4]{Lai11}.

\begin{thm} \label{thm:kaw} 
Assume the existence of good models in dimension $n-q$. Let $X$ be a minimal terminal projective variety of dimension $n$. If $\kappa (X,K_X) = q$, then $K_X$ is semiample. 
\end{thm} 

As promised in the introduction, we show the equivalence of Mumford's conjecture and the weaker Conjecture \ref{mum2}.

\begin{pro}\label{pro:mumford}
Assume that Conjecture \ref{mum2} holds in dimensions at most $n$. Then Conjecture \ref{mum} holds in dimension $n$.
\end{pro}

\begin{proof}
We follow closely the proof of \cite[Proposition IV.5.7]{Kol96}. Let $X$ be a projective manifold of dimension $n$ such that 
$$ H^0\big(X,(\Omega^1_X)^{\otimes m}\big) = 0\quad\text{for all } m \geq 1.$$
Then $X$ is uniruled by Conjecture \ref{mum2} and by \cite{BDPP}, and let $\pi\colon X \dashrightarrow Z$ be an MRC fibration of $X$, see \cite[\S IV.5]{Kol96}. By blowing up $X$ and $Z$, we may additionally assume that $\pi$ is a morphism. By \cite[Corollary 1.4]{GHS03}, $Z$ is not uniruled. If $X$ is not rationally connected, then $\dim Z\geq1$ and $K_Z$ is pseudoeffective by \cite{BDPP}, hence
$$ H^0\big(Z,(\Omega^1_Z)^{\otimes m_0}\big) \neq 0 $$
for some positive integer $m_0$ by Conjecture \ref{mum2}. Since
$$(\pi^*\Omega_Z^1)^{\otimes m_0}\subseteq(\Omega_X^1)^{\otimes m_0},$$
we obtain $H^0\big(X,(\Omega^1_X)^{\otimes m_0}\big) \neq 0 $, a contradiction.
\end{proof}

\begin{rem}\label{rem:tensor}
We often use without explicit mention that any effective tensor representation of a vector bundle $\mathcal E$ on a variety $X$ can be embedded as a submodule in its high tensor power, see \cite[Chapter III, \S6.3 and \S7.4]{Bou98}. In particular, if $H^0\big(X,(\bigwedge^q\mathcal E)^{\otimes p}\big)\neq0$ for some $p,q>0$, then $H^0\big(X,\mathcal E^{\otimes m}\big) \neq 0 $ for some $m>0$.
\end{rem}

We finish the section by commenting on log and singular cases. 

\begin{rem} 
(1) Let  $(X,\Delta)$ be a projective klt pair and let $\pi\colon \widehat X \to X$ be a log resolution.  Assume that $-(K_X + \Delta)$ nef and that 
$$ H^0\big(\widehat X, \big(\Omega^1_{\widehat X}\big)^{\otimes m}\big) = 0 $$
for all positive integers $m$. Then $\widehat X$ and $X$ are rationally connected. Indeed, if $K_{\widehat X}$ is pseudoeffective, then $X$ has canonical singularities and $\Delta = 0$. In this case $K_X\sim_\Q0$ and hence $ H^0\big(\widehat X, \big(\Omega^1_{\widehat X}\big)^{\otimes m}\big) \neq0 $ for some $m$, contradicting our assumption. Therefore $K_{\widehat X}$ is not pseudoeffective and $\widehat X$ is uniruled by \cite{BDPP}. Let $\widehat X \dasharrow Z$ be an MRC fibration to a projective manifold $Z$. By \cite[Main Theorem]{Zh05} we have $\kappa (Z,K_Z) = 0$, and we conclude that $\dim Z= 0$ as in the proof of Proposition \ref{pro:mumford}. Therefore $\widehat X$ as well as $X$ are rationally connected. 

(2) In a singular setting, one might hope to characterize rational connectedness using reflexive differentials. However, \cite[Example 3.7]{GKP14} constructs a rational surface $X$ with only rational double points such that 
$$ H^0\big(X,((\Omega^1_X)^{\otimes 2})^{**}\big) \neq 0.$$ 
\end{rem} 

\section{Numerical dimension 1}\label{sec:nd1}

The basis of this section is the following result \cite[Theorem 6.7]{LP16}.

\begin{thm}\label{thm:nu1}
Let $X$ be a minimal $\Q$-factorial projective terminal variety such that $\nu(X,K_X)=1$. If $\chi(X,\OO_X)\neq0$, then $\kappa(X,K_X)\geq0$.
\end{thm}

We give some comments on the proof of Theorem \ref{thm:nu1} in \cite{LP16}. Assuming for contradiction that $\kappa(X,K_X)=-\infty$, the 
main step is to show that then for a resolution $\pi\colon Y\to X$, for all $m \neq 0$ sufficiently divisible and for all $p$ we have
\begin{equation}\label{eq:20}
H^0\big(Y,\Omega^p_Y \otimes \OO_Y(m\pi^*K_X)\big)=0.
\end{equation}
There are two crucial inputs here: the first one is the birational stability of the cotangent bundle \cite{CP11,CP15}; the second is a criterion which says that if a nef Cartier divisor $L$ 
with $\nu(X,L)=1$ can be written as $L=P+D$, where $P$ is a pseudoeffective divisor and $D\neq0$ is an effective divisor, then $\kappa(X,L)\geq0$. 

Now we distinguish two cases: if there exists a positive integer $m$ such that $\pi^*\OO_X(mK_X)$ has a singular metric $h_m$ such that the multiplier ideal $\mathcal I(h_m)$ does 
not equal $\OO_Y$, then one uses the criterion above to conclude; note that in this case the assumption $\chi(X,\OO_X) \neq0$ is not needed. Otherwise, for each $m$ there is a  
singular metric $h_m$ as above such that $\mathcal I(h_m)=\OO_Y$, and then the Hard Lefschetz theorem from \cite{DPS01} gives
$$H^q\big(Y,\OO_Y(K_Y+m\pi^*K_X)\big)=0\quad\text{for all }q.$$
An easy argument involving the Hirzebruch-Riemann-Roch allows to conclude $\chi(X,\OO_X) = 0$, which gives a contradiction.

\medskip

The theorem implies quickly the main result of this section.

\begin{thm} \label{thm:nu11}
Let $X$ be a projective manifold such that $K_X$ is pseudo-effective, and assume that $X$ has a minimal model. If $\nu(X,K_X) =1$, then there exists a positive integer $m$ such that 
$$ H^0\big(X,(\Omega^1_X)^{\otimes m}\big) \neq 0. $$
In particular, Conjecture \ref{mum2} holds if $\dim X=4$ and $\nu(X,K_X)=1$.
\end{thm} 

\begin{proof} 
Assume to the contrary that 
$$ H^0\big(X,(\Omega^1_X)^{\otimes m}\big) = 0\quad\text{for all } m \geq 1,$$
so that, in particular,
$$ H^q(X,\OO_X)\simeq H^0(X,\Omega^q_X) = 0\quad\text{for all } q \geq 1.$$ 
Therefore $ \chi(X,\OO_X) = 1$. If $Y$ is a minimal model of $X$, then $\nu(Y,K_Y)=1$ and $\chi(Y,\OO_Y) = 1$, hence $\kappa(X,K_X)=\kappa(Y,K_Y)\geq0$ by Theorem \ref{thm:nu1}. This is a contradiction. 

The second statement follows immediately since minimal models of ca\-no\-ni\-cal fourfolds exist by \cite{BCHM,Fuj05}.
\end{proof} 

\section{Numerical codimension 1} 

\begin{thm} \label{thm:nu-n}  
Let $X$ be a minimal terminal $n$-fold with $\nu(X,K_X) = n-1$ and $n \geq 4$, and let $\pi\colon Y \to X$ be a resolution which is an isomorphism over the smooth locus of $X$. Assume one of the following:
\begin{enumerate}
\item[(a)] $(\pi^*K_X)^{n-2}  \cdot c_2(Y) \neq 0$;
\item[(b)] $(\pi^*K_X)^{n-2}  \cdot c_2(Y) = 0$ and $(\pi^*K_X)^{n-3} \cdot K_Y \cdot c_2(Y) \ne 0$. 
\end{enumerate} 
Then $K_X$ is semiample. 
\end{thm} 

\begin{proof} 
Since $X$ has terminal singularities, the singular locus of $X$ is of dimension at most $n-3$ by \cite[Corollary 5.18]{KM98}, hence
\begin{equation}\label{eq:dot}
(\pi^*K_X)^n=(\pi^*K_X)^{n-1} \cdot K_Y = (\pi^*K_X)^{n-2} \cdot K_Y^2 = 0.
\end{equation}
Let $m$ be any positive integer such that $mK_X$ is Cartier. Then by Hirze\-bruch-Riemann-Roch, by Serre duality, by \eqref{eq:dot} and since $X$ has rational singularities, we obtain
\begin{align}
\chi\big(Y,\OO_Y(K_Y &+ \pi^*(mK_X))\big) =  (-1)^n\chi\big(Y,\OO_Y(-\pi^*(mK_X))\big)\label{eq:RR}\\
&= \frac{1}{12(n-2)!} m^{n-2} (\pi^*K_X)^{n-2} \cdot c_2(Y)\notag \\
&+ \frac{1}{24(n-3)!} m^{n-3} (\pi^*K_X)^{n-3} \cdot K_Y \cdot c_2(Y) + O(m^{n-4})\notag
\end{align}
and
\begin{align}
\chi(X,\OO_X(m&K_X))= \chi\big(Y,\OO_Y(\pi^*(mK_X))\big)\label{eq:RR2}\\
&= \frac{1}{12(n-2)!} m^{n-2} (\pi^*K_X)^{n-2} \cdot c_2(Y)\notag \\
&-\frac{1}{24(n-3)!} m^{n-3} (\pi^*K_X)^{n-3} \cdot K_Y \cdot c_2(Y) + O(m^{n-4}).\notag
\end{align}
By Miyaoka's inequality \cite[\S6]{Miy87}, we have
$$ (\pi^*K_X)^{n-2} \cdot c_2(Y)\geq0.$$ 

Suppose first that $(\pi^*K_X)^{n-2} \cdot c_2(Y) > 0$. Since 
\begin{equation}\label{eq:vanish}
H^i\big(X,\OO_X(mK_X)\big) = 0\quad\text{ for }i\geq2 
\end{equation}
by Lemma \ref{lem:KVvanishing}, by \eqref{eq:RR2} there exists a constant $C_1>0$ such that 
$$ h^0\big(X,\OO_X(mK_X)\big) \geq C_1m^{n-2}. $$ 
We conclude that $\kappa (X,K_X) \geq n-2$, hence $K_X$ is semiample by Theorem \ref{thm:kaw}.

From now on suppose that 
$$(\pi^*K_X)^{n-2}  \cdot c_2(Y) = 0\quad\text{and}\quad (\pi^*K_X)^{n-3} \cdot K_Y \cdot c_2(Y) \neq 0.$$ 
If 
$$ (\pi^*K_X)^{n-3} \cdot K_Y \cdot c_2(Y)  < 0, $$
then \eqref{eq:RR2} and \eqref{eq:vanish} imply that there exists a constant $C_2>0$ such that 
$$ h^0\big(X,\OO_X(mK_X)\big) \geq C_2m^{n-3}, $$ 
and $K_X$ is semiample by Theorem \ref{thm:kaw}. If
$$ (\pi^*K_X)^{n-3}  \cdot K_Y \cdot c_2(Y) > 0,$$
then since 
$$ H^i\big(X,\OO_Y(K_Y+\pi^*(mK_X))\big) = 0\quad\text{ for }i\geq2 $$
by Lemma \ref{lem:KVvanishing}, by \eqref{eq:RR} there exists a constant $C_3>0$ such that 
\begin{equation}\label{eq:22}
h^0\big(Y,\OO_Y(K_Y + \pi^*(mK_X))\big) \geq C_3m^{n-3}.
\end{equation}
We claim that then $\kappa(X,K_X)\geq n-3$, and therefore that $K_X$ is semiample as before. Indeed, by \eqref{eq:22} there exists a positive integer $m_0$ and an effective divisor $D$ such that $K_Y + \pi^*(m_0K_X)\sim D$, hence $(m_0+1)K_X\sim_\Q \pi_*D$ and $\kappa(X,K_X)\geq0$. Fix a positive integer $p$ such that $pK_X$ is Cartier and $h^0(X,pK_X)>0$. Then \eqref{eq:22} gives
\begin{multline*}
h^0\big(X,\OO_X(2pmK_X)\big)\geq h^0\big(X,\OO_X(p(m+1)K_X)\big)\\
=h^0\big(Y,\OO_Y(pK_Y + \pi^*(pmK_X))\big)\geq C_3m^{n-3}= C_4(2pm)^{n-3},
\end{multline*}
where $C_4=C_3/(2p)^{n-3}$. This finishes the proof.
\end{proof} 

In dimension $n = 4$ we obtain more precisely: 

\begin{thm} \label{thm:nu3}  
Let $X$ be a minimal terminal $4$-fold with $\nu(X,K_X) = 3$, and let $\pi\colon Y \to X$ be a resolution which is an isomorphism over the smooth locus of $X$. Assume one of the following:
\begin{enumerate}
\item[(a)] $(\pi^*K_X)^2  \cdot c_2(Y) \neq 0$;
\item[(b)] $(\pi^*K_X)^2  \cdot c_2(Y) = 0$ and $\chi(X,\OO_X) > 0$.
\end{enumerate} 
Then $\kappa (X,K_X) \geq 0$. 
\end{thm} 

\begin{proof} 
By Theorem \ref{thm:nu-n}, we may assume that 
$$(\pi^*K_X)^2 \cdot c_2(X) =  0\quad\text{and}\quad (\pi^*K_X) \cdot K_Y \cdot c_2(Y) =  0.$$ 
In this case Hirzebruch-Riemann-Roch gives
$$ \chi\big(Y,\OO_Y(K_Y + m \pi^*(mK_X)\big) = \chi(Y,\OO_Y) = \chi(X,\OO_X)$$
and 
$$ \chi\big(X,\OO_X(mK_X)\big) = \chi(X,\OO_X).$$ 
Hence, arguing as in the proof of Theorem \ref{thm:nu-n} we obtain $\kappa (X,K_X) \geq 0$. Note that we can no longer apply Theorem \ref{thm:kaw} to deduce semiampleness. 
\end{proof} 

\begin{rem} 
Let $X$ be a minimal terminal $n$-fold with $\nu(X,K_X) = n-1$, and let $\pi\colon Y \to X$ be a resolution which is an isomorphism over the smooth locus of $X$. We argue that if we have the vanishing 
\begin{equation}\label{eq:21}
(\pi^*K_X)^{n-2}  \cdot c_2(Y) = 0,
\end{equation}
then the geometry of $X$ is special. Indeed, assume that $K_X$ is semiample, and let $f\colon X \to Z$ be the associated Iitaka fibration. We claim that $f$ is almost smooth in codimension one. Indeed, there is a positive integer $m$ and a very ample divisor $A$ on $Z$ such that $mK_X = f^*A$. If $D_1, \dots,D_{n-2} \in \vert A \vert $ are general elements, then $C = D_1 \cap \ldots \cap D_{n-2}$ is a smooth curve and $S = f^{-1}(C)$ is a smooth surface proportional to $K_X^{n-2}$, hence \eqref{eq:21} implies
$$ c_2(T_X \vert_S) = c_2(T_S) = 0. $$ 
Hence the only singular fibres of $f \vert_S$ are multiple elliptic curves (this is a classical fact: use Proposition V.12.2 and Remark before that proposition in \cite{BHPV04} together with Hirzebruch-Riemann-Roch). Consequently, there is a subset $B \subseteq Z$ of codimension at least $2$ such that 
for each $b \in X \setminus B$, the fibre of $f$ over $b$ is a smooth elliptic curve or a multiple of an elliptic curve. 
\end{rem} 

Finally, we obtain the proof of Conjecture \ref{mum2} on $4$-folds of numerical dimension $3$.

\begin{cor} 
Let $X$ be a smooth projective $4$-fold with $K_X$ pseudoeffective and $\nu(X,K_X) = 3$. Then there is a positive integer $m$ such that 
$$ H^0\big(X,(\Omega^1_X)^{\otimes m}\big) \neq 0.$$ 
\end{cor} 

\begin{proof}  
Assume to the contrary that 
$$ H^0\big(X,(\Omega^1_X)^{\otimes m}\big) = 0\quad\text{for all } m \geq 1,$$
so that $ \chi(X,\OO_X) = 1$ as in the proof of Theorem \ref{thm:nu11}. Let $Y$ be a minimal model of $X$, which exists by \cite{BCHM,Fuj05}. Then $\nu(Y,K_Y)=3$ and $\chi(Y,\OO_Y) = 1$, hence $\kappa (Y,K_Y) \geq 0$ by Theorem \ref{thm:nu3}. This is a contradiction. 
\end{proof} 

\section{Metrics with algebraic singularities} 

We recall the definition of a singular hermitian metric with algebraic singularities, following \cite{DPS01} and \cite{Dem01}.

\begin{dfn} 
Let $X$ be a normal projective variety and let $D$ be a $\Q$-Cartier divisor. We say that $D$, or $\mathcal O_X(D)$, has a \emph{metric with algebraic singularities 
and semipositive curvature current}, if there exists a positive integer $m$ such that $mD$ is Cartier and if there exists a resolution of singularities $\pi\colon Y \to X$ such that the line bundle $\pi^*\OO_X(D)$ has a singular metric $h$ whose curvature current is semipositive (as a current), and the local plurisubharmonic weights $\varphi$ of $h$ are of the form
$$ \varphi = \sum \lambda_j \log \vert g_j \vert + O(1),$$
where $\lambda_j$ are positive rational numbers, $O(1)$ is a bounded term, and the divisors $D_j$ defined locally by $g_j$ form a simple normal crossing divisor on $Y$. 
\end{dfn} 

\begin{rem}\label{rem:metric}
It is well known that a line bundle $L$ on a normal projective variety is pseudoeffective if and only if it has a singular metric whose curvature current is semipositive. It is a consequence of the Minimal Model Program that on a terminal variety with the pseudoeffective canonical sheaf, the canonical sheaf always has a metric with algebraic singularities and semipositive curvature current.
\end{rem}

The following is one of the main results of \cite{LP16}.

\begin{thm} \label{thm:LP4.3} 
Assume the existence of good minimal models of klt pairs in dimensions at most $n-1$ and let $X$ be a projective terminal variety of dimension $n$ with $K_X$ pseudoeffective. Suppose that $K_X$ has a metric with algebraic singularities and semipositive curvature current. If $\chi(X,\OO_X) \neq 0$, then $\kappa(X,K_X) \geq 0$. 
\end{thm} 

Towards Mumford's conjecture, this implies:

\begin{thm} \label{thm4} 
Let $X$ be a projective manifold of dimension $n$ with $K_X$ pseudoeffective. Assume that $K_X$ has a metric  with algebraic singularities and semipositive curvature current.
\begin{enumerate}
\item[(i)] If good minimal models for klt pairs exist in dimensions at most $n-1$, then there is a positive integer $m$ such that
$$ H^0\big(X,(\Omega^1_X)^{\otimes m}\big) \neq 0. $$
\item[(ii)] If $n=4$, then there is a positive integer $m$ such that
$$ H^0\big(X,(\Omega^1_X)^{\otimes m}\big) \neq 0. $$
\end{enumerate} 
\end{thm} 

\begin{proof}
Assume to the contrary that 
$$ H^0\big(X,(\Omega^1_X)^{\otimes m}\big) = 0\quad\text{for all } m \geq 1,$$
so that $ \chi(X,\OO_X) = 1$ as in the proof of Theorem \ref{thm:nu11}. Then $\kappa(X,K_X)\geq0$ by Theorem \ref{thm:LP4.3}, which is a contradiction that shows (i). 

The second statement follows immediately since good models of ca\-no\-ni\-cal threefolds exist by \cite{Mor88,Sho85,Miy87,Miy88a,Miy88b,Kaw92}.
\end{proof}

\bibliographystyle{amsalpha}

\bibliography{biblio}
\end{document}

%% file: macros.tex
\theoremstyle{plain}

\newtheorem{thm}{Theorem}[section]

\newtheorem{pro}[thm]{Proposition}

\newtheorem{lem}[thm]{Lemma}

\newtheorem{cor}[thm]{Corollary}
\newtheorem{con}[thm]{Conjecture}

\theoremstyle{definition}

\newtheorem{dfn}[thm]{Definition}

\newtheorem{rem}[thm]{Remark}

\theoremstyle{remark}

\newcommand{\N}{\mathbb{N}}

\newcommand{\Q}{\mathbb{Q}}

\newcommand{\OO}{\mathcal{O}}

